\documentclass[a4paper,11pt]{amsart}

\usepackage{mathptmx,amssymb,amscd,latexsym,eulervm}
\usepackage{amsmath}
\usepackage{amsthm}
\usepackage{amsfonts}
\usepackage[onehalfspacing]{setspace}
\usepackage{paralist}
\usepackage{bm}
\usepackage{aliascnt}
\usepackage[initials,lite]{amsrefs}
\usepackage[inner=2.4cm,outer=2.4cm,bottom=2.9cm]{geometry}
\usepackage{xcolor}
\definecolor{citepink}{HTML}{AA3377}
\usepackage[colorlinks=true,citecolor=citepink,linkcolor=blue,urlcolor=blue]{hyperref}

\allowdisplaybreaks
\AtBeginDocument{\def\MR#1{}}

\makeatletter
\@namedef{subjclassname@2020}{\textup{2020} Mathematics Subject Classification}
\makeatother

\newtheorem{theorem}{Theorem}[section]
\newtheorem*{theoremA}{Theorem A}
\newtheorem*{theoremB}{Theorem B}
\newaliascnt{proposition}{theorem}
\newtheorem{proposition}[proposition]{Proposition}
\aliascntresetthe{proposition}
\newaliascnt{corollary}{theorem}
\newtheorem{corollary}[corollary]{Corollary}
\aliascntresetthe{corollary}
\newaliascnt{lemma}{theorem}
\newtheorem{lemma}[lemma]{Lemma}
\aliascntresetthe{lemma}
\theoremstyle{remark}
\newaliascnt{remark}{theorem}
\newtheorem{remark}[remark]{Remark}
\aliascntresetthe{remark}
\theoremstyle{definition}
\newtheorem{example}[theorem]{Example}

\DeclareMathOperator{\Hom}{Hom}
\DeclareMathOperator{\End}{End}
\DeclareMathOperator{\tr}{tr}
\DeclareMathOperator{\indeg}{indeg}

\newcommand{\kk}{\Bbbk}

\newcommand{\ZZ}{\mathbb Z}
\newcommand{\m}{\mathfrak m}

\title[Two-sided bounds for canonical traces of Veronese subalgebras]{Two-sided bounds for canonical traces of Veronese subalgebras}
\author{Sora Miyashita}
\address[Miyashita]{Department of Pure and Applied Mathematics, Graduate School of Information Science and Technology, The University of Osaka, Suita, Osaka 565-0871, Japan}
\email{u804642k@ecs.osaka-u.ac.jp}
\date{\today}
\hypersetup{pdftitle={Canonical Traces and Anti-Canonical Modules of Veronese Subalgebras},pdfauthor={Sora Miyashita}}

\begin{document}
\emergencystretch=1em

\begin{abstract}
The classical Veronese formula of Goto--Watanabe identifies the canonical module of a Veronese subalgebra with the corresponding Veronese module. We prove the analogous compatibility for graded duals under a module-finite extension whose upper ring is equidimensional, satisfies Serre's condition $(S_2)$, and has dimension at least two, and obtain the formula for the $b$-invariant. This yields explicit two-sided bounds for canonical traces: the lower bound is controlled by the Loewy length of the canonical-trace quotient, whereas the upper bound is governed by the $a$- and $b$-invariants, with an $S_2$-ification refinement for unmixed rings. We also characterize when all sufficiently large Veronese subalgebras of standard graded level algebras are nearly Gorenstein when the $a$-invariant is negative. Applications include Segre products, Stanley--Reisner rings and determinantal rings.
\end{abstract}

\maketitle

\section{Introduction}

Passage to Veronese subalgebras is a basic operation in graded commutative algebra. Let $R=\bigoplus_{i\ge0}R_i$ be a Noetherian positively graded algebra over a field $\kk$. For $d\ge1$, its \emph{$d$-th Veronese subalgebra} is $R^{(d)}=\bigoplus_{i\ge0}R_{id}$. A classical theorem of Goto--Watanabe gives $\omega_{R^{(d)}}\cong(\omega_R)^{(d)}$; see~\cite{GW}*{Corollary~(3.1.3)}. Since the canonical trace is the image of the evaluation map $\omega_R\otimes_R\Hom_R(\omega_R,R)\to R$, its behavior under the Veronese operation depends not only on the canonical module but also on its graded dual, or \emph{anti-canonical module}. Our first result supplies the corresponding compatibility on the dual side.

For a finite graded $R$-module $M$, its \emph{trace} is $\tr_R(M)=\sum_{\varphi\in\Hom_R(M,R)}\varphi(M)$. For a finite graded module $N$, its \emph{initial degree} is $\indeg N=\min\{j\in\ZZ:N_j\ne0\}$, with $\indeg0=\infty$. Following Goto--Watanabe~\cite{GW}*{Definition~(2.1.2)}, let $\omega_R$ denote the graded canonical module of $R$, and set $a(R)=-\indeg\omega_R$ and $b(R)=\indeg\Hom_R(\omega_R,R)$. The \emph{$b$-invariant}, introduced by Okuma--Watanabe--Yoshida~\cite{OWY} in the normal graded setting, is the initial degree of the anti-canonical module. For a Cohen--Macaulay ring, $\tr_R(\omega_R)$ describes the non-Gorenstein locus, and $R$ is \emph{nearly Gorenstein} when its graded maximal ideal lies in $\tr_R(\omega_R)$; see~\cite{HHS}*{Lemma~2.1 and Definition~2.2}. See also \cites{Lindo,HKS,KM,MV,Miyazaki,JSSZ}.

\begin{theoremA}[see Theorem~\ref{thm:dual}]
Let $A\subseteq B$ be a degree-preserving inclusion of Noetherian positively graded $\kk$-algebras such that $B$ is finitely generated as an $A$-module, with $A$ standard graded. Assume that $B$ is equidimensional, satisfies $(S_2)$, and has dimension at least two. Then, for every $d\ge1$,
\[
\Hom_{B^{(d)}}(\omega_{B^{(d)}},B^{(d)})\cong\Hom_B(\omega_B,B)^{(d)},
\qquad
b(B^{(d)})=\left\lceil\frac{b(B)}d\right\rceil.
\]
\end{theoremA}

The proof obtains this extension by a codimension-one argument using $(S_2)$. The second ingredient is Proposition~\ref{prop:transfer}, which is valid for arbitrary finite graded modules. Together these results give the following estimate. Let $\m_d$ denote the graded maximal ideal of $R^{(d)}$.

\begin{theoremB}[see Theorem~\ref{thm:main}]
Let $R$ be a standard graded $\kk$-algebra
of dimension at least two
satisfying $(S_2)$, with graded maximal ideal $\m$. Set $\ell=\min\{j\ge0:\m^j\subseteq\tr_R(\omega_R)\}$ and assume that $\ell<\infty$. Then, for every $d\ge1$,
\[
\m_d^{\,1+\lceil(\ell-1)/d\rceil}\subseteq\tr_{R^{(d)}}(\omega_{R^{(d)}})\subseteq\m_d^{\,\lceil-a(R)/d\rceil+\lceil b(R)/d\rceil}.
\]
\end{theoremB}

The integer $\ell$ is the Loewy length of $R/\tr_R(\omega_R)$, so $\ell<\infty$ is equivalent to the canonical trace having finite colength. Since canonical modules and trace ideals commute with localization, \cite{Aoyama}*{(1.7) and Corollary~4.3} and \cite{AG}*{Proposition~3.3} show that, under $(S_2)$, this forces $R$ to be equidimensional and quasi-Gorenstein on the punctured spectrum; see also~\cite{KM}*{Remark~2.9~(2), (5)}. In the Cohen--Macaulay case, this is precisely the condition that $R$ be Gorenstein on the punctured spectrum. The lower inclusion in Theorem~B requires neither $(S_2)$ nor unmixedness, whereas the upper inclusion comes from Theorem~A and the initial degrees of the canonical and anti-canonical modules. For unmixed rings, Theorem~\ref{thm:main} refines the upper bound using the $b$-invariant of the finite graded $S_2$-ification introduced in Section~\ref{sec:trace}; Example~\ref{ex:non-s2} shows that the sharp bound involving $b(R)$ itself can fail without $(S_2)$.

When $\ell\le1$, Theorem~B recovers the positive-dimensional Veronese stability theorem for nearly Gorenstein Cohen--Macaulay algebras and, by the same argument, also gives the Artinian case. Independently of the finite-colength hypothesis, Theorem~A yields an exact formula for the degree-one canonical trace of Veronese subalgebras of level algebras; for $a(R)<0$, Corollary~\ref{cor:level-large} gives a necessary and sufficient criterion for all sufficiently large Veronese subalgebras to be nearly Gorenstein. Section~\ref{sec:example} applies the results to Segre products, Stanley--Reisner rings and determinantal rings.

The paper is organized as follows. Section~\ref{sec:duality} proves Theorem~A and the level-algebra criterion. Section~\ref{sec:trace} proves Theorem~B, and the $S_2$-ification refinement. Section~\ref{sec:example} contains the applications.

\section{Canonical modules and graded duals}\label{sec:duality}

In this section, we fix the graded conventions used in the paper, prove the compatibility of anti-canonical modules with Veronese subalgebras, and apply it to level algebras. A \emph{positively graded algebra} $T=\bigoplus_{i\ge0}T_i$ will mean a Noetherian graded $\kk$-algebra with $T_0$ finite-dimensional over $\kk$; it is \emph{standard graded} when $T_0=\kk$ and $T$ is generated by $T_1$. Following Goto--Watanabe~\cite{GW}*{Definition~(2.1.2)}, and componentwise over the Artinian decomposition of $T_0$ when necessary, we write $\omega_T$ for its graded canonical module. Their Veronese canonical-module formula is understood in the same componentwise sense. The integer $a(T)=-\indeg\omega_T$ is the \emph{$a$-invariant} of $T$, and $Q(T)$ denotes its total quotient ring. For a graded $T$-module $M$ and $s\in\ZZ$, our shift convention is $[M(s)]_i=M_{i+s}$. For finite graded $T$-modules $M$ and $N$, we use the natural grading
$\Hom_T(M,N)_j=\{\varphi\in\Hom_T(M,N):\varphi(M_i)\subseteq N_{i+j}\text{ for all }i\}$.

For a graded module $M$ and integers $d\ge1$ and $0\le r<d$, set $M^{\langle d,r\rangle}=\bigoplus_{j\in\ZZ}M_{jd+r}$ and $M^{(d)}=M^{\langle d,0\rangle}$, with the usual regrading. We use the standard compatibility
$\omega_{R^{(d)}}\cong(\omega_R)^{(d)}$
of graded canonical modules with Veronese subalgebras; see \cite{GW}*{Corollary~(3.1.3)}. In particular, when $R$ is standard graded, positive-dimensional, equidimensional, and satisfies $(S_2)$,
\[
a(R^{(d)})=\left\lfloor\frac{a(R)}d\right\rfloor.
\]
Indeed, after a faithfully flat extension of the ground field, a linear nonzerodivisor is also $\omega_R$-regular by the graded form of~\cite{Aoyama}*{(1.7)}, so the nonzero graded components of $\omega_R$ occur in every degree at least $\indeg\omega_R$.
First of all, we show the following.

\begin{lemma}\label{lem:intersection}
Let $R$ be a Noetherian ring satisfying $(S_2)$. For each $\mathfrak p\in\operatorname{Spec}R$, let $\lambda_{\mathfrak p}\colon Q(R)\longrightarrow Q(R_{\mathfrak p})$ denote the canonical map. Then, 
$R=\bigcap_{\operatorname{ht}\mathfrak p=1}\lambda_{\mathfrak p}^{-1}(R_{\mathfrak p})$,
where $R_{\mathfrak p}$ is regarded as a subring of $Q(R_{\mathfrak p})$.
\end{lemma}

This is the standard codimension-one intersection property associated with $(S_2)$. The following proof is essentially the same as \cite{MatsumuraCA}*{\S17, Theorem~38 and (17.I)}, but we include it for the reader's convenience.

\begin{proof}
Let $z=a/s\in Q(R)$, where $s$ is a nonzerodivisor, and assume that $\lambda_{\mathfrak p}(z)\in R_{\mathfrak p}$ for every height-one prime $\mathfrak p$. Suppose that $a\notin sR$, and set $M=R(a+sR)\subseteq R/sR$. Then $M\ne0$, so choose $\mathfrak p\in\operatorname{Ass}_R M$. Since $M\subseteq R/sR$, one has $\mathfrak p\in\operatorname{Ass}_R(R/sR)$. $R/sR$ satisfies $(S_1)$, so $\mathfrak p$ is minimal over $sR$. Krull's principal ideal theorem gives $\operatorname{ht}\mathfrak p\le1$. Since $s$ is a nonzerodivisor, $\mathfrak p$ is not a minimal prime of $R$; hence $\operatorname{ht}\mathfrak p=1$.

By the choice of $\mathfrak p$, one has $M_{\mathfrak p}\ne0$. On the other hand, $\lambda_{\mathfrak p}(z)\in R_{\mathfrak p}$ implies $a/1\in sR_{\mathfrak p}$. Thus $(a+sR)_{\mathfrak p}=0$, and therefore $M_{\mathfrak p}=0$, a contradiction.
\end{proof}

\begin{proposition}\label{prop:residue}
Let $A\subseteq B$ be a degree-preserving inclusion of Noetherian positively graded $\kk$-algebras such that $B$ is finitely generated as an $A$-module. Assume that $A$ is standard graded, that $B$ is equidimensional of dimension at least two and satisfies $(S_2)$, and that $A_1$ contains a nonzerodivisor on $B$. Then, for every finite graded $B$-module $M$, every $d\ge1$, every $0\le r<d$, and every $n\in\ZZ$, restriction induces an isomorphism
$\Hom_{B^{(d)}}(M^{\langle d,r\rangle},B^{(d)})_n\cong[\Hom_B(M,B)]_{nd-r}$.
\end{proposition}

\begin{proof}
Let $x\in A_1$ be a nonzerodivisor on $B$. Restriction gives the displayed map. It is injective. Indeed, suppose that a homogeneous map $\psi\colon M\to B$ vanishes on $M^{\langle d,r\rangle}$. For a homogeneous element $u\in M$, choose $q\ge0$ such that $\deg u+q\equiv r\pmod d$. Then $x^qu\in M^{\langle d,r\rangle}$, and hence $x^q\psi(u)=\psi(x^qu)=0$. Since $x$ is a nonzerodivisor on $B$, it follows that $\psi(u)=0$.

For surjectivity, let $\varphi\colon M^{\langle d,r\rangle}\to B^{(d)}$ be homogeneous of degree $n$. For a homogeneous element $u\in M$, choose $q\ge0$ such that $\deg u+q\equiv r\pmod d$, and set
$\psi_x(u)=x^{-q}\varphi(x^qu)\in B_x$.
This definition is independent of the choice of $q$. Indeed, if $q'\ge q$ is another choice, then $q'-q$ is a nonnegative multiple of $d$, and $B^{(d)}$-linearity gives
$x^{-q'}\varphi(x^{q'}u)=x^{-q'}x^{q'-q}\varphi(x^qu)=x^{-q}\varphi(x^qu)$.

We next verify $B$-linearity. Let $b\in B$ and $u\in M$ be homogeneous, and let $q$ be admissible for $u$. Choose $q'\ge q$ such that $\deg b+q'-q$ is a nonnegative multiple of $d$. Then $q'$ is admissible for $bu$, and $bx^{q'-q}\in B^{(d)}$. Therefore
\[
\psi_x(bu)=x^{-q'}\varphi(x^{q'}bu)=x^{-q'}\varphi\bigl(bx^{q'-q}x^qu\bigr)=x^{-q'}bx^{q'-q}\varphi(x^qu)=b\psi_x(u).
\]
Extending additively, we obtain a homogeneous $B$-linear map $\psi_x\colon M\to B_x$ of degree $nd-r$.

Let $\mathfrak p$ be a height-one prime of $B$. We claim that $A_1\nsubseteq\mathfrak p$. Suppose otherwise. Since $A$ is standard graded, one has $\mathfrak p\cap A=A_+$. Consequently, $B/\mathfrak p$ is module-finite over $A/A_+=\kk$. Since $B/\mathfrak p$ is a domain, it is a field, and hence $\mathfrak p$ is maximal.

Choose a minimal prime $\mathfrak q\subseteq\mathfrak p$ of $B$. Since $B$ is equidimensional, $\dim(B/\mathfrak q)=\dim B$. Moreover, $B/\mathfrak q$ is an affine domain over $\kk$, and $\mathfrak p/\mathfrak q$ is maximal. Hence
$\operatorname{ht}_{B/\mathfrak q}(\mathfrak p/\mathfrak q)=\dim(B/\mathfrak q)=\dim B$.
It follows that $\operatorname{ht}_B\mathfrak p\ge\dim B\ge2$, contrary to $\operatorname{ht}_B\mathfrak p=1$. Thus we may choose $y_{\mathfrak p}\in A_1\setminus\mathfrak p$.

Using $y_{\mathfrak p}$ in place of $x$, the same construction defines a $B_{\mathfrak p}$-linear map $\psi_{\mathfrak p}\colon M_{\mathfrak p}\to B_{\mathfrak p}$ whose restriction to $M^{\langle d,r\rangle}$ agrees with $\varphi$. Explicitly, for a homogeneous $u\in M$ and $q\ge0$,
$\psi_{\mathfrak p}(u/1)=y_{\mathfrak p}^{-q}\varphi(y_{\mathfrak p}^qu)\in B_{\mathfrak p}$.

The $B_x$-module $M_x$ is generated by the image of $M^{\langle d,r\rangle}$. Indeed, for every homogeneous $u\in M$, one may choose $q\ge0$ such that $x^qu\in M^{\langle d,r\rangle}$, and then $u=x^{-q}(x^qu)$ in $M_x$. Hence, after localizing at both $\mathfrak p$ and $x$, the maps $\psi_x$ and $\psi_{\mathfrak p}$ agree:
$(\psi_x)_{\mathfrak p}=(\psi_{\mathfrak p})_x\colon M_{\mathfrak p,x}\longrightarrow B_{\mathfrak p,x}$.
Since $x$ is a nonzerodivisor on $B_{\mathfrak p}$, it follows that, for every homogeneous $u\in M$,
$\lambda_{\mathfrak p}(\psi_x(u))=\psi_{\mathfrak p}(u/1)\in B_{\mathfrak p}$
inside $Q(B_{\mathfrak p})$.

This holds for every height-one prime $\mathfrak p$ of $B$. Lemma~\ref{lem:intersection} therefore gives $\psi_x(u)\in B$ for every homogeneous $u\in M$. Thus $\psi_x$ is a homogeneous $B$-linear map $M\to B$ of degree $nd-r$, and its restriction to $M^{\langle d,r\rangle}$ is $\varphi$.
\end{proof}

We prove the main result of this paper below.

\begin{theorem}\label{thm:dual}
Let $A\subseteq B$ be a degree-preserving inclusion of Noetherian positively graded $\kk$-algebras such that $B$ is finitely generated as an $A$-module, with $A$ standard graded. Assume that $B$ is equidimensional, satisfies $(S_2)$, and has dimension at least two. Then, for every $d\ge1$,
$\Hom_{B^{(d)}}(\omega_{B^{(d)}},B^{(d)})\cong\Hom_B(\omega_B,B)^{(d)}$
and
$b(B^{(d)})=\left\lceil\frac{b(B)}{d}\right\rceil$,
with the convention that $\lceil\infty/d\rceil=\infty$.
\end{theorem}

\begin{proof}
By faithfully flat base change, we may assume that $\kk$ is infinite; see \cite{HHS}*{Lemma~1.5~(iii)}.

Since $B$ satisfies $(S_2)$, one has $\operatorname{Ass}_B B=\operatorname{Min}B$. No minimal prime of $B$ contains $A_1$. Indeed, suppose that $\mathfrak q\in\operatorname{Min}B$ contains $A_1$. Then $\mathfrak q\cap A=A_+$, so $B/\mathfrak q$ is module-finite over $A/A_+=\kk$. Hence $\mathfrak q$ is maximal and $\dim(B/\mathfrak q)=0$, contradicting equidimensionality and $\dim B\ge2$.

Thus $A_1$ is not contained in any associated prime of $B$. Since $\kk$ is infinite, we may choose
$x\in A_1\setminus\bigcup_{\mathfrak q\in\operatorname{Ass}_B B}\mathfrak q$.
Then $x$ is a nonzerodivisor on $B$. Proposition~\ref{prop:residue}, applied to $M=\omega_B$ and $r=0$, gives
\[
\Hom_{B^{(d)}}((\omega_B)^{(d)},B^{(d)})\cong\Hom_B(\omega_B,B)^{(d)}.
\]
Together with Goto--Watanabe's formula $\omega_{B^{(d)}}\cong(\omega_B)^{(d)}$ from \cite{GW}*{Corollary~(3.1.3)}, this proves the first assertion.

Set $H=\Hom_B(\omega_B,B)$. Suppose first that $H\ne0$, and write $b(B)=\min\{j\in\ZZ\mid H_j\ne0\}$. Multiplication by $x$ is injective on $H$, since $x$ is a nonzerodivisor on $B$. Therefore $H_j\ne0$ for every $j\ge b(B)$. Using the first assertion and the usual regrading of the Veronese module, we obtain
\[
b(B^{(d)})=\min\{n\in\ZZ : H_{nd}\ne0\}=\left\lceil\frac{b(B)}{d}\right\rceil.
\]
If $H=0$, then the first assertion shows that $\Hom_{B^{(d)}}(\omega_{B^{(d)}},B^{(d)})=0$, so both sides are infinite by convention.
\end{proof}

\begin{remark}
The dimension assumption in Theorem~\ref{thm:dual} cannot be dropped. Let $R=\kk[x]$ and $C=R^{(2)}=\kk[x^2]$, where $C$ is equipped with the usual Veronese regrading. Then
$\omega_R=R(-1)$
and
$\omega_C=C(-1)\cong(\omega_R)^{(2)}$.
Nevertheless,
$\Hom_C(\omega_C,C)\cong C(1)$,
$\Hom_R(\omega_R,R)^{(2)}\cong R(1)^{(2)}\cong C$.
Thus the two graded modules have initial degrees $-1$ and $0$, respectively, and hence are not isomorphic. Under these identifications, the restriction map is $C\to C(1)$, $f\mapsto x^2f$, and its cokernel is $\kk(1)$.
\end{remark}

The notion of a level algebra was introduced by Stanley; see~\cite{StanleyCCA}*{Chapter~III, Section~3}. A standard graded Cohen--Macaulay algebra is \emph{level} if its graded canonical module is generated in a single degree. Applying Theorem~\ref{thm:dual}, we obtain a criterion for the nearly Gorenstein property of Veronese subalgebras of level algebras without any punctured-Gorenstein assumption; see Corollary~\ref{cor:level-large}.

\begin{proposition}\label{prop:level-trace}
Let $R$ be a standard graded level algebra of dimension at least two, and fix $d\ge1$. Set $c=-a(R)$ and $e=\lceil c/d\rceil$. Then
$[\tr_{R^{(d)}}(\omega_{R^{(d)}})]_1
=
[\omega_R]_{de}
[\Hom_R(\omega_R,R)]_{d(1-e)}$,
where the product is taken via evaluation. Consequently, $R^{(d)}$ is nearly Gorenstein precisely when the product on the right is $R_d$. If $c>0$ and $d\ge c$, then
$[\tr_{R^{(d)}}(\omega_{R^{(d)}})]_1
=
R_{d-c}[\tr_R(\omega_R)]_c$.
\end{proposition}

\begin{proof}
Set $H=\Hom_R(\omega_R,R)$. Since $R$ is level, $\omega_R=R[\omega_R]_c$. For $n\ge e$, standard gradedness gives
\[
[\omega_R]_{nd}
=R_{nd-c}[\omega_R]_c
=R_{(n-e)d}R_{ed-c}[\omega_R]_c
=R_{(n-e)d}[\omega_R]_{ed}.
\]
Hence the isomorphism $\omega_{R^{(d)}}\cong(\omega_R)^{(d)}$ shows that $\omega_{R^{(d)}}$ is generated in degree $e$. By Theorem~\ref{thm:dual}, the degree-$(1-e)$ component of its graded dual is $H_{d(1-e)}$. Taking the degree-one part of the evaluation map gives
$[\tr_{R^{(d)}}(\omega_{R^{(d)}})]_1
=[\omega_R]_{de}H_{d(1-e)}$.
Since $R^{(d)}$ is Cohen--Macaulay and standard graded, it is nearly Gorenstein precisely when this component equals $R_d$.

Suppose now that $d\ge c>0$. Then $e=1$ and $[\omega_R]_d=R_{d-c}[\omega_R]_c$. Moreover,
$[\tr_R(\omega_R)]_c=[\omega_R]_cH_0$.
Indeed, every contribution $[\omega_R]_{c+j}H_{-j}$ to the degree-$c$ trace is contained in $[\omega_R]_cH_0$, because $[\omega_R]_{c+j}=R_j[\omega_R]_c$, and the reverse inclusion is the term $j=0$. The final formula follows.
\end{proof}

\begin{corollary}\label{cor:level-large}
Let $R$ be a standard graded level algebra of dimension at least two with $a(R)<0$. Set $J=[\tr_R(\omega_R)]_{-a(R)}R$. For $d\ge-a(R)$, $R^{(d)}$ is nearly Gorenstein precisely when $J_d=R_d$. Hence all sufficiently large Veronese subalgebras are nearly Gorenstein precisely when $\sqrt J=\m$.
\end{corollary}

\begin{proof}
For $d\ge-a(R)$, Proposition~\ref{prop:level-trace} gives $[\tr_{R^{(d)}}(\omega_{R^{(d)}})]_1=J_d$. Since $R^{(d)}$ is Cohen--Macaulay and standard graded, it is nearly Gorenstein precisely when this degree-one component is $R_d$. The last assertion follows because $J_d=R_d$ for all sufficiently large $d$ precisely when
$\sqrt J=\m$.
\end{proof}

\begin{remark}
Corollary~\ref{cor:level-large} complements
\cite{MiyVer}*{Theorem~1.2~(1)}: the latter gives a sufficient criterion without assuming levelness, whereas the former gives a necessary and sufficient criterion for level algebras with $a(R)<0$ in terms of $[\tr_R(\omega_R)]_{-a(R)}$.
\end{remark}

\section{Canonical traces of Veronese subalgebras}\label{sec:trace}

This section proves the two-sided estimates for canonical traces of Veronese subalgebras. We first establish the lower bound, then derive the upper bounds from graded duality and the $S_2$-ification, and finally record a counterexample and a general consequence.

We first record the statement used for the lower bound of Theorem~B.

\begin{proposition}\label{prop:transfer}
Let $S$ be a Noetherian positively graded algebra, let $M$ be a finite graded $S$-module, and set $\mathfrak n=(S_1)S$. If $I\subseteq\tr_S(M)$ is a graded ideal, then
$\tr_{S^{(d)}}(M^{(d)})
\supseteq
(\mathfrak n^{d-1}I)^{(d)}$
for every $d\ge1$.
\end{proposition}

\begin{proof}
It is enough to consider a homogeneous term $x_1\cdots x_{d-1}\theta$, where $x_i\in S_1$, $\theta\in I$, and $\deg\theta\equiv1\pmod d$. Write
$\theta=\sum_j\varphi_j(u_j)$
with homogeneous $u_j$ and homogeneous maps $\varphi_j:M\to S$. Choose $r_j\in\{0,\ldots,d-1\}$ such that $\deg u_j+r_j\in d\ZZ$. Splitting $x_1\cdots x_{d-1}$ after $r_j$ factors moves $u_j$ into $M^{(d)}$, while multiplication of $\varphi_j$ by the remaining factors gives an $S^{(d)}$-linear map to $S^{(d)}$. Summing these maps yields the desired element of the trace.
\end{proof}

\begin{remark}
Proposition~\ref{prop:transfer} extends the positive-dimensional statement in~\cite{MiyVer}*{Theorem~3.4}.
\end{remark}

\begin{corollary}\label{cor:artinian}
Let $A$ be an Artinian standard graded nearly Gorenstein algebra. Then $A^{(d)}$ is nearly Gorenstein for every $d\ge1$.
\end{corollary}

\begin{proof}
Let $\m_A$ be the graded maximal ideal. Since $\m_A\subseteq\tr_A(\omega_A)$, Proposition~\ref{prop:transfer} gives
\[
\tr_{A^{(d)}}((\omega_A)^{(d)})
\supseteq
(\m_A^d)^{(d)}
=
\m_{A^{(d)}}.
\]
Since $\omega_{A^{(d)}}\cong(\omega_A)^{(d)}$, the assertion follows.
\end{proof}

\begin{corollary}\label{cor:all-dim}
Let $R$ be a standard graded Cohen--Macaulay algebra. If $R$ is nearly Gorenstein, then $R^{(d)}$ is nearly Gorenstein for every $d\ge1$.
\end{corollary}

\begin{proof}
For positive dimension this is~\cite{MiyVer}*{Theorem~5.6~(2)~(b)}; dimension zero is Corollary~\ref{cor:artinian}.
\end{proof}

We next separate the lower estimate from the upper estimate and pass to the $S_2$-ification. We use \emph{unmixed} to mean that $\dim R/\mathfrak p=\dim R$ for every $\mathfrak p\in\operatorname{Ass}R$; in particular, an unmixed ring is equidimensional.
Let $R$ be a standard graded unmixed ring and let $\m=R_+$. By a \emph{graded $S_2$-ification} of $R$ we mean a finite graded $R$-algebra $S$ such that
$R\subseteq S\subseteq Q(R)$,
$S$ satisfies $(S_2)$, and $S_{\mathfrak p}=R_{\mathfrak p}$ for every homogeneous prime $\mathfrak p$ of $R$ with $\operatorname{ht}\mathfrak p\le1$. From this point on, set
$\widetilde R:=\End_R(\omega_R)$
with its natural grading.

\begin{remark}\label{rem:s2ification}
We now recall why $\widetilde R$ is the graded $S_2$-ification of $R$ in the above sense. Since $\omega_R$ is finite, $\widetilde R$ is a finite graded $R$-algebra. For every homogeneous prime $\mathfrak p\subseteq\m$, localization of graded Hom gives
$\widetilde R_{\mathfrak p}\cong\End_{R_{\mathfrak p}}((\omega_R)_{\mathfrak p})$.
The local ring $R_{\m}$ is equidimensional and unmixed.
Hence the natural map $R\to\widetilde R$ is injective, and $\widetilde R$ may be identified with a graded subring of $Q(R)$; see~\cite{HH}*{(2.1), (2.2)~(e), (f)}. Then \cite{Aoyama}*{Corollary~4.3 and Theorem~3.2} show that every homogeneous localization of $\widetilde R$ satisfies $(S_2)$ and that $\omega_R$, with its natural $\widetilde R$-module structure, is a graded canonical module of $\widetilde R$. If $\operatorname{ht}\mathfrak p\le1$, then $R_{\mathfrak p}$ is already $(S_2)$, so~\cite{AG}*{Proposition~1.2} gives $\widetilde R_{\mathfrak p}=R_{\mathfrak p}$. Thus $\widetilde R$ is the finite graded $S_2$-ification of $R$; compare~\cite{HH}*{(2.2), (2.5)--(2.7)}. Since $(\widetilde R)_0$ is finite-dimensional over $\kk$, Theorem~\ref{thm:dual} applies to the finite graded extension $R\subseteq\widetilde R$.
\end{remark}

\begin{lemma}\label{lem:conductor}
Let $R$ be a standard graded unmixed algebra with graded maximal ideal $\m$. Then
$\tr_R(\omega_R)\subseteq(R:\widetilde R)$.
Consequently, if $\m^\ell\subseteq\tr_R(\omega_R)$ for some $\ell\ge0$, then
$\m^\ell(\widetilde R/R)=0$, and hence $\widetilde R/R$ has finite length.
\end{lemma}
\begin{proof}
Since $R_{\m}$ is equidimensional and unmixed, $(\omega_R)_{\m}$ is faithful by~\cite{HH}*{(2.2)~(e)}, so the hypothesis $U_{R_{\m}}(0)=0$ in~\cite{AG}*{(3.8)} is satisfied. After localizing at $\m$, that inclusion gives the desired containment. Trace and conductor commute with localization, and both sides are graded ideals, so the containment holds over $R$. The remaining assertions are immediate.
\end{proof}

\begin{proposition}\label{prop:upper-hull}
Let $R$ be a standard graded unmixed generically Gorenstein algebra of dimension at least two. Then, for every $d\ge1$,
\[
\tr_{R^{(d)}}(\omega_{R^{(d)}})\subseteq\m_d^{\,\lceil-a(R)/d\rceil+\lceil b(\widetilde R)/d\rceil}.
\]
If $\m^\ell\subseteq\tr_R(\omega_R)$, then $b(\widetilde R)\ge b(R)-\ell$, and consequently
\[
\tr_{R^{(d)}}(\omega_{R^{(d)}})\subseteq\m_d^{\,\max\{0,\,\lceil-a(R)/d\rceil+\lceil(b(R)-\ell)/d\rceil\}}.
\]
\end{proposition}

\begin{proof}
After a faithfully flat extension we may assume that $\kk$ is infinite. By~\cite{AG}*{(3.1)}, applied to $R_{\m}$, and using the grading, choose a homogeneous canonical ideal $I\subseteq R$ containing a nonzerodivisor and an integer $s$ such that $\omega_R\cong I(s)$, and write $s=ad+r$ with $0\le r<d$. By Remark~\ref{rem:s2ification}, $\omega_R$ is also a graded canonical module of $\widetilde R$, with its natural $\widetilde R$-module structure. Goto--Watanabe's formula therefore identifies both $\omega_{R^{(d)}}$ and $\omega_{\widetilde R^{(d)}}$, up to the same shift, with $I^{\langle d,r\rangle}$.

Let $\varphi:I^{\langle d,r\rangle}\to R^{(d)}$ be homogeneous. Choose a linear nonzerodivisor $x\in R_1$. After multiplying a homogeneous nonzerodivisor of $I$ by a suitable power of $x$, choose a homogeneous nonzerodivisor $c\in I^{\langle d,r\rangle}$. As in the proof of Proposition~\ref{prop:residue}, cross-multiplication after multiplying by $x^\delta$, where $\delta=0$ if $r=0$ and $\delta=d-r$ otherwise, shows that $\varphi$ is multiplication by a homogeneous element $z\in Q(R)$. Since $I$ is a $\widetilde R$-module, $I^{\langle d,r\rangle}$ is a $\widetilde R^{(d)}$-module; multiplication by $z$ is therefore $\widetilde R^{(d)}$-linear. Its image lies in $R^{(d)}\subseteq\widetilde R^{(d)}$, and hence
$\Hom_{R^{(d)}}(\omega_{R^{(d)}},R^{(d)})
\subseteq
\Hom_{\widetilde R^{(d)}}(\omega_{\widetilde R^{(d)}},\widetilde R^{(d)})$.
Theorem~\ref{thm:dual} gives $b(\widetilde R^{(d)})=\lceil b(\widetilde R)/d\rceil$. Since $R$ is unmixed, the graded form of~\cite{Aoyama}*{(1.7)} gives $\operatorname{Ass}_R\omega_R=\operatorname{Assh}R=\operatorname{Ass}R$, so the chosen linear nonzerodivisor $x$ is also $\omega_R$-regular. Multiplication by $x$ therefore shows that $(\omega_R)_j\ne0$ for every $j\ge\indeg\omega_R=-a(R)$. Together with $\omega_{R^{(d)}}\cong(\omega_R)^{(d)}$, this gives $\indeg\omega_{R^{(d)}}=\lceil-a(R)/d\rceil$. Thus the source and the graded dual in the evaluation map have initial degrees at least $\lceil-a(R)/d\rceil$ and $\lceil b(\widetilde R)/d\rceil$, respectively. Since $R^{(d)}$ is standard graded, the evaluation description of the trace gives the first containment.

Assume now that $\m^\ell\subseteq\tr_R(\omega_R)$. Choose the linear form $x$ above to be a nonzerodivisor on $R$; it is then also a nonzerodivisor on $\widetilde R\subseteq Q(R)$. If $0\ne\psi\in\Hom_{\widetilde R}(\omega_R,\widetilde R)_{b(\widetilde R)}$, Lemma~\ref{lem:conductor} gives $\m^\ell\widetilde R\subseteq R$, so $x^\ell\psi\ne0$ lies in $\Hom_R(\omega_R,R)$ in degree $b(\widetilde R)+\ell$. Thus $b(R)\le b(\widetilde R)+\ell$.
\end{proof}

Before stating the sharp upper estimate, we recall the $b$-invariant introduced by Okuma--Watanabe--Yoshida~\cite{OWY}*{Section~2}. For a finitely generated normal graded domain $T$ with fractional canonical ideal $K_T$, they define $b(T)=\indeg K_T^{-1}$; in the present setting, we use $b(T)=\indeg\Hom_T(\omega_T,T)$, which agrees with their definition in the normal-domain case.

\begin{proposition}\label{prop:upper}
Let $R$ be a standard graded equidimensional algebra of dimension at least two satisfying $(S_2)$, with graded maximal ideal $\m$. Then, for every $d\ge1$,
$b(R^{(d)})=\left\lceil\frac{b(R)}d\right\rceil$
and
$$\tr_{R^{(d)}}(\omega_{R^{(d)}})
\subseteq
\m_d^{\,\lceil-a(R)/d\rceil+\lceil b(R)/d\rceil}.$$
\end{proposition}

\begin{proof}
The equality for $b(R^{(d)})$ is Theorem~\ref{thm:dual} applied with $A=B=R$. Since $\indeg\omega_{R^{(d)}}=\lceil-a(R)/d\rceil$, the evaluation description of the trace~\cite{Miyazaki}*{Lemma~2.2} gives the containment.
\end{proof}

Under the hypotheses of Theorem~B, the quotient $R/\tr_R(\omega_R)$ has finite length. At every minimal prime $\mathfrak p$, localization gives $\tr_{R_{\mathfrak p}}((\omega_R)_{\mathfrak p})=R_{\mathfrak p}$, and hence $(\omega_R)_{\mathfrak p}\ne0$. Thus every minimal prime belongs to $\operatorname{Supp}_R\omega_R$, and~\cite{Aoyama}*{(1.7)} shows that $R$ is equidimensional. Moreover, for every prime $\mathfrak p\ne\m$, the module $(\omega_R)_{\mathfrak p}$ is canonical by~\cite{Aoyama}*{Corollary~4.3}, and~\cite{AG}*{Proposition~3.3} shows that $R_{\mathfrak p}$ is quasi-Gorenstein; see also~\cite{KM}*{Remark~2.9~(2), (5)}. Thus Proposition~\ref{prop:upper} applies, and Theorem~B follows from Propositions~\ref{prop:upper} and~\ref{prop:transfer}. We now record the stronger form valid for unmixed rings.

\begin{theorem}\label{thm:main}
Let $R$ be a standard graded unmixed algebra of dimension at least two, with graded maximal ideal $\m$. Assume that $\ell=\min\{j\ge0:\m^j\subseteq\tr_R(\omega_R)\}<\infty$. Then, for every $d\ge1$,
\[
\m_d^{\,1+\lceil(\ell-1)/d\rceil}\subseteq\tr_{R^{(d)}}(\omega_{R^{(d)}})\subseteq\m_d^{\,\lceil-a(R)/d\rceil+\lceil b(\widetilde R)/d\rceil}.
\]
Moreover $b(\widetilde R)\ge b(R)-\ell$, so the right-hand side may be weakened to $\m_d^{\,\max\{0,\lceil-a(R)/d\rceil+\lceil(b(R)-\ell)/d\rceil\}}$. If $R$ satisfies $(S_2)$, then $\widetilde R=R$ and the sharp upper exponent is $\lceil-a(R)/d\rceil+\lceil b(R)/d\rceil$.
\end{theorem}

\begin{proof}
The inclusion $\m^\ell\subseteq\tr_R(\omega_R)$ says that $R/\tr_R(\omega_R)$ has finite length. Since $R$ is equidimensional and unmixed, $\omega_R$ is faithful by~\cite{HH}*{(2.2)~(e)}, and $(\omega_R)_{\mathfrak p}$ is a canonical module of $R_{\mathfrak p}$ for every prime $\mathfrak p$ by~\cite{Aoyama}*{Corollary~4.3}. For every nonmaximal prime $\mathfrak p$, localization gives $\tr_{R_{\mathfrak p}}((\omega_R)_{\mathfrak p})=R_{\mathfrak p}$, so~\cite{AG}*{Proposition~3.3} shows that $R_{\mathfrak p}$ is quasi-Gorenstein. In particular, $R_{\mathfrak p}$ is Gorenstein for every minimal prime $\mathfrak p$. Thus $R$ is generically Gorenstein, and Proposition~\ref{prop:upper-hull} gives the upper estimates.

For the lower bound, apply Proposition~\ref{prop:transfer} with $M=\omega_R$ and $I=\m^\ell$. Since $R$ is standard graded, $\m^q=\bigoplus_{j\ge q}R_j$ for every $q\ge0$. Hence
$(\m^{d-1+\ell})^{(d)}
=\m_d^{\,\lceil(d-1+\ell)/d\rceil}
=\m_d^{\,1+\lceil(\ell-1)/d\rceil}$.
The final assertion follows from $\widetilde R=R$ in the $S_2$ case, equivalently from Proposition~\ref{prop:upper}.
\end{proof}

\begin{remark}
The lower estimate in Theorem~\ref{thm:main} requires neither $(S_2)$ nor unmixedness; those hypotheses enter only through the upper estimates. If the lower exponent agrees with either upper exponent, then the canonical trace is exactly the corresponding power of $\m_d$.
\end{remark}

\begin{example}\label{ex:non-s2}
The $S_2$ hypothesis cannot simply be removed from the sharp upper bound involving $b(R)$. Let
$R=\kk[s^4,s^3t,st^3,t^4]\subseteq\kk[s,t]$,
where the four displayed generators have degree one, and let
$S=\kk[s,t]^{(4)}=R[s^2t^2]$.
Then $S/R\cong\kk(-1)$, because $R_n=S_n=\kk[s,t]_{4n}$ for every $n\ge2$. Thus $S$ is the $S_2$-ification of the two-dimensional domain $R$, while $R$ itself does not satisfy $(S_2)$. By Remark~\ref{rem:s2ification}, $\omega_R$ is a canonical module of $S$; equivalently here, $\omega_R\cong\omega_S$ and $[\omega_R]_n=\kk[s,t]_{4n-2}$. Thus $a(R)=-1$.

A homogeneous map $\omega_R\to R$ of degree zero would be multiplication by an element $z\in\kk[s,t]_2$ satisfying
$z\kk[s,t]_2\subseteq R_1
=\langle s^4,s^3t,st^3,t^4\rangle_\kk$.
Writing $z=\alpha s^2+\beta st+\gamma t^2$ and multiplying successively by $s^2,st,t^2$ forces $\alpha=\beta=\gamma=0$. On the other hand, multiplication by every element of $\kk[s,t]_6$ gives a degree-one map $\omega_R\to R$. Hence $b(R)=1$. Its degree-two trace component is
$\kk[s,t]_6\kk[s,t]_2=\kk[s,t]_8=R_2$,
and therefore
$\tr_R(\omega_R)=\m^2,
\ell=2$.
For $d=2$ one has
$R^{(2)}=S^{(2)}=\kk[s,t]^{(8)}$,
and multiplication $\kk[s,t]_6\kk[s,t]_2=\kk[s,t]_8$ shows that
$\tr_{R^{(2)}}(\omega_{R^{(2)}})=\m_2$.
Thus the formula obtained by deleting $(S_2)$ from Proposition~\ref{prop:upper} would incorrectly give
$\tr_{R^{(2)}}(\omega_{R^{(2)}})\subseteq\m_2^2$,
because $\lceil-a(R)/2\rceil+\lceil b(R)/2\rceil=2$. By contrast, $b(\widetilde R)=b(S)=0$, so Proposition~\ref{prop:upper-hull} gives the correct exponent one.
\end{example}

\begin{corollary}\label{cor:quadratic}
Let $R$ be a standard graded equidimensional algebra of dimension at least two satisfying $(S_2)$, and assume that $\ell<\infty$. If $a(R)<0<b(R)$, then
$\tr_{R^{(d)}}(\omega_{R^{(d)}})\subseteq\m_d^2$
for every $d\ge1$. In particular, when $R^{(d)}$ is Cohen--Macaulay, it is not nearly Gorenstein. Moreover,
$\tr_{R^{(d)}}(\omega_{R^{(d)}})=\m_d^2$
for all sufficiently large $d$.
\end{corollary}

\begin{proof}
By Proposition~\ref{prop:upper}, both ceiling terms in the upper exponent are at least one. Taking $d=1$ gives $\tr_R(\omega_R)\subseteq\m^2$, and hence $\ell\ge2$. For all sufficiently large $d$, both ceiling terms are one and $1+\lceil(\ell-1)/d\rceil=2$. The two estimates therefore coincide.
\end{proof}

\section{Applications to classes with explicit canonical traces}\label{sec:example}

In this section, we apply the preceding results to Segre products, Stanley--Reisner rings and determinantal rings.
We introduce the notation needed for each class before stating the applications.

\subsection{Segre products}

For positively graded $\kk$-algebras $B$ and $C$, their \emph{Segre product} is
\[
B\# C=\bigoplus_{j\ge0}B_j\otimes_\kk C_j,
\]
with the grading induced by the displayed decomposition.

\begin{corollary}\label{cor:segre}
Let $B$ and $C$ be standard graded Gorenstein domains of dimension at least two with $a(B)<0$ and $a(C)<0$, and assume that $T=B\# C$ is Cohen--Macaulay. Then $T^{(d)}$ is nearly Gorenstein for every $d\ge-\min\{a(B),a(C)\}$.
\end{corollary}

\begin{proof}
By \cite{GW}*{Theorem~(4.3.1)}, $\omega_T\cong\omega_B\#\omega_C$. Since $B$ and $C$ are Gorenstein, this shows that $T$ is level with $a(T)=\min\{a(B),a(C)\}$. Set $q=|a(B)-a(C)|$. Herzog--Hibi--Stamate~\cite{HHS}*{Theorem~4.15~(i)} give $\m_T^q\subseteq\tr_T(\omega_T)$. Since $a(B),a(C)<0$, one has $q\le-a(T)$, and therefore $[\tr_T(\omega_T)]_{-a(T)}=T_{-a(T)}$. Set $J=[\tr_T(\omega_T)]_{-a(T)}T$. Then $J_d=T_d$ for every $d\ge-a(T)$, and the assertion follows from Corollary~\ref{cor:level-large}.
\end{proof}

\begin{remark}
If $T$ is a domain, then \cite{HHS}*{Theorem~4.15~(ii)} strengthens the preceding inclusion to $\tr_T(\omega_T)=\m_T^{|a(B)-a(C)|}$.
\end{remark}

\subsection{Stanley--Reisner rings}
Let $\Delta$ be a simplicial complex on $\{1,\ldots,n\}$, let $S=\kk[x_1,\ldots,x_n]$, and let
\[
I_\Delta
=
\bigl(
x_{i_1}\cdots x_{i_q}
:
1\le q\le n,\ 
1\le i_1<\cdots<i_q\le n,\ 
\{i_1,\ldots,i_q\}\notin\Delta
\bigr)
\subseteq S
\]
be its \emph{Stanley--Reisner ideal}. The quotient $\kk[\Delta]=S/I_\Delta$ is the \emph{Stanley--Reisner ring} of $\Delta$, and its graded maximal ideal is $\m=(x_1,\ldots,x_n)\kk[\Delta]$.

\begin{corollary}\label{cor:sr}
Let $R=\kk[\Delta]$ be Cohen--Macaulay of dimension at least three and Gorenstein on the punctured spectrum. 
Assume that $R$ is not Gorenstein.
Then, for every $d\ge1$,
$\tr_{R^{(d)}}(\omega_{R^{(d)}})=\m_d^2$.
In particular, no Veronese subalgebra of $R$ is nearly Gorenstein.
\end{corollary}
\begin{proof}
By~\cite{MV}*{Theorem~A}, one has $\tr_R(\omega_R)=\m^2$. Moreover, \cite{MV}*{Theorem~B and Proposition~3.3~(2)} show that $R$ is level and that $\omega_R=R[\omega_R]_1$; hence $a(R)=-1$. Since every Stanley--Reisner ring is reduced, $R$ is generically Gorenstein.

By~\cite{AG}*{(3.1)}, applied to $R_{\m}$, and using the grading, identify $\omega_R$ with a homogeneous fractional canonical ideal $K$ satisfying $K=RK_1$. After a faithfully flat extension of the ground field, a general element of $K_1$ is $R$-regular. Multiplication by this element is injective on $K^{-1}$, and therefore $\indeg(KK^{-1})=1+b(R)$. Since $KK^{-1}=\tr_R(\omega_R)=\m^2$, we obtain $b(R)=1$. The lower and upper exponents in Theorem~B are therefore both two for every $d\ge1$.
\end{proof}

\subsection{Determinantal rings}

Let $S=\kk[x_{ij}:1\le i\le m,\ 1\le j\le n]$, let $X=(x_{ij})$ be the generic $m\times n$ matrix, and denote by $I_t(X)\subseteq S$ the ideal generated by its $t\times t$ minors.

\begin{corollary}\label{cor:det}
Assume $1\le r<m\le n$, and set $R=S/I_{r+1}(X)$. Then
for every $d\ge1$,
\[
\tr_{R^{(d)}}(\omega_{R^{(d)}})
\subseteq
\m_d^{\,\lceil-a(R)/d\rceil+\lceil(a(R)+r(n-m))/d\rceil}.
\]
If $r=1$, then this containment is an equality; more precisely,
for every $d\ge1$,
\[
\tr_{R^{(d)}}(\omega_{R^{(d)}})
=
\m_d^{\,\lceil n/d\rceil-\lfloor m/d\rfloor}.
\]
\end{corollary}

\begin{proof}
By~\cite{FHST}*{Theorem~1.1},
$\tr_R(\omega_R)=I_r(X)^{n-m}R$. Since $R$ is a domain,
$\indeg\tr_R(\omega_R)=-a(R)+b(R)=r(n-m)$,
and hence $b(R)=a(R)+r(n-m)$. The stated containment follows from
Proposition~\ref{prop:upper}.

Suppose now that $r=1$. Let $P=\kk[u_1,\ldots,u_m]$ and
$Q=\kk[v_1,\ldots,v_n]$. Under the standard identification
$R\cong P\#Q$, which sends $x_{ij}$ to $u_iv_j$, one has
$[\omega_R]_j=P_{j-m}\otimes_\kk Q_{j-n}$,
$[\Hom_R(\omega_R,R)]_j
=
P_{j+m}\otimes_\kk Q_{j+n}$;
the second description follows from~\cite{HHS}*{Proposition~4.18}.
Thus $[\tr_{R^{(d)}}(\omega_{R^{(d)}})]_d \neq 0$ when there
exists $i\in\ZZ$ such that $id\ge n$ and $(t-i)d\ge-m$. Whenever
such an $i$ exists, multiplication in $P$ and $Q$ is surjective, so
this component equals $R_{td}$. The least possible $t$ is
\[
\left\lceil\frac nd\right\rceil
+
\left\lceil-\frac md\right\rceil
=
\left\lceil\frac nd\right\rceil
-
\left\lfloor\frac md\right\rfloor.
\]
Since $R^{(d)}$ is standard graded, the asserted equality follows.
\end{proof}


\section*{Acknowledgments}
The author was supported by JSPS KAKENHI Grant Number 25KJ1744. OpenAI's ChatGPT was used during preparation as an auxiliary generative-AI tool for exploratory discussion, checks of algebraic manipulations and examples, comparison with cited literature, and assistance with English exposition and LaTeX. All material used in the final manuscript was checked by the author against the relevant arguments and sources. The author assumes full responsibility for the mathematics, citations, and final text.

\begin{singlespace}

\end{singlespace}

\end{document}